\documentclass[11pt,reqno]{amsart}
\usepackage[T1]{fontenc}
\usepackage{graphicx}
\usepackage{cite}
\usepackage{fixltx2e}
\usepackage[utf8]{inputenc}

\usepackage{color}
\definecolor{MyLinkColor}{rgb}{0,0,0.4}
\numberwithin{equation}{section}

\newcommand{\re}{\mathop{\rm Re}\nolimits}

\newcommand{\0}{\Omega}
\newcommand{\e}{\varepsilon}

\newcommand{\p}{\partial}
\newcommand{\wt}{\widetilde}

\newcommand{\G}{\Gamma}

\newcommand{\bA}{\mathbb{A}}

\newcommand{\cO}{\mathcal{O}}

\newcommand{\kH}{\mathcal{H}}

\newcommand{\kL}{\mathcal{L}}

\newcommand{\E}{\mathbb{E}}

\newcommand{\R}{\mathbb{R}}
\newcommand{\s}{\mathbb S}
\newcommand{\N}{\mathbb{N}}

\newtheorem{thm}{Theorem}[section]

\newtheorem{lemma}[thm]{Lemma}

\theoremstyle{remark} 
\newtheorem{rem}[thm]{Remark}

\numberwithin{equation}{section}

\title[Mean curvature flow of rotationally symmetric closed surfaces]{On a degenerate parabolic system describing the mean curvature flow of rotationally symmetric closed surfaces}

  \author{Harald Garcke}
\address{Fakult\"at f\"ur Mathematik, Universit\"at Regensburg,   93040 Regensburg, Deutschland.}
\email{harald.garcke@ur.de}
  
\author{Bogdan--Vasile Matioc}
\address{Fakult\"at f\"ur Mathematik, Universit\"at Regensburg,   93040 Regensburg, Deutschland.}
\email{bogdan.matioc@ur.de}


\usepackage[colorlinks=true,linkcolor=MyLinkColor,citecolor=MyLinkColor]{hyperref} %dvips

\begin{document}

\begin{abstract}
 We show that the mean curvature flow for a closed and rotationally symmetric surface  can be formulated as an evolution problem consisting of an evolution equation for the square
  of the function whose graph is rotated and  
two ODEs describing the evolution of the points of the evolving surface that lie on the rotation axis. 
For the fully nonlinear and degenerate  parabolic problem we establish the well-posedness property in the setting of classical solutions.
Besides we prove that the problem features the effect of parabolic smoothing.
\end{abstract}

\maketitle

 %%%%%%%%%%%%%%%%%%%%%%%%%%%%%%%%%%%%%%%%%%%%%%%%%%%%%%%%%%%%%%%%%%%
%%%%%%%%%%%%%%%%%%%%%%%%%%%%%%%%%%%%%%%%%%%%%%%%%%%%%%%%%%%%%%%%%%%%
%%%%%%%%%%%%%%%%%%%%%%%%%%%%%%%%%%%%%%%%%%%%%%%%%%%%%%%%%%%%%%%%%%%%
%%%%%%%%%%%%%%%%%%%%%%%%%%%%%%%%%%%%%%%%%%%%%%%%%%%%%%%%%%%%%%%%%%%
%%%%%%%%%%%%%%%%%%%%%%%%%%%%%%%%%%%%%%%%%%%%%%%%%%%%%%%%%%%%%%%%%%%%
%%%%%%%%%%%%%%%%%%%%%%%%%%%%%%%%%%%%%%%%%%%%%%%%%%%%%%%%%%%%%%%%%%%%
 
\section{Introduction}\label{Sec:1}
 %%%%%%%%%%%%%%%%%%%%%%%%%%%%%%%%%%%%%%%%%%%%%%%%%%%%%%%%%%%%%%%%%%%
%%%%%%%%%%%%%%%%%%%%%%%%%%%%%%%%%%%%%%%%%%%%%%%%%%%%%%%%%%%%%%%%%%%%
%%%%%%%%%%%%%%%%%%%%%%%%%%%%%%%%%%%%%%%%%%%%%%%%%%%%%%%%%%%%%%%%%%%%
%%%%%%%%%%%%%%%%%%%%%%%%%%%%%%%%%%%%%%%%%%%%%%%%%%%%%%%%%%%%%%%%%%%
%%%%%%%%%%%%%%%%%%%%%%%%%%%%%%%%%%%%%%%%%%%%%%%%%%%%%%%%%%%%%%%%%%%%
%%%%%%%%%%%%%%%%%%%%%%%%%%%%%%%%%%%%%%%%%%%%%%%%%%%%%%%%%%%%%%%%%%%%
Mean curvature flow is the most efficient way to decrease the surface
area of a hypersurface. It hence has been of great interest in geometry
as well as in materials science and image analysis, see
\cite{Mu56,Br78, Hu84, CDR03, Ec04, Gi06, Ma11}. 
Since the pioneering work of Brakke \cite{Br78} and Huisken \cite{Hu84} many results have been
shown for mean curvature flow and we refer to \cite{Ma11}  
 and the references therein for more information about the subject. 
The case of  rotationally symmetric evolutions lead to 
spatially one-dimensional problems  and  due to the reduced complexity
this situation has been studied by several authors analytically \cite{DK91, Hu90,
AltschulerAG95, MB07,
LeCrone14} as well as numerically \cite{MayerS02, BGN19}. 

In particular,  rotationally symmetric mean curvature flow has been
helpful
to understand singularity formation in  curvature flows, see
\cite{Hu90, DK91,
AltschulerAG95, MB07, EM10}. 
Most of the analytical results have been
restricted to the case of surfaces with boundary or periodic unbounded
situations. The situation becomes analytically far more involved
if one considers closed surfaces, i.e., compact  surfaces without
boundaries. 
In this  context the governing equation can be recast, provided that the 
points on the rotation axis have positive curvature, as a free boundary problem which involves both degenerate and
singular terms.
 This paper gives  first well-posedness and parabolic smoothing results for the 
 free boundary problem describing compact rotationally symmetric  surfaces evolving by mean curvature  derived herein. 
 
Let us now precisely formulate the analytic problem.
We study the  evolution of a family of rotationally symmetric   surfaces  $\{\G(t)\}_{t\geq0}$  by the mean curvature flow. 
Given $t\geq0$, we assume that $$\G(t):=\{(x,u(t,x)\cos\vartheta, u(t,x)\sin\vartheta)\,:\, x\in[a(t),b(t)],\, \vartheta\in[0,2\pi]\}\subset\R^3$$
is the surface obtained by rotating the graph of the unknown function $u(t):[a(t),b(t)]\to\R$  around the $x$-axis.
Moreover, we consider herein the case when the surfaces $\G(t)$ are closed, meaning in particular that also the domain of definition $[a(t),b(t)]$ of $u(t)$ is unknown.
Since the motion of the surfaces is governed by the equation
\begin{align}\label{MCF}
V(t)=H(t)\quad\text{on $\Gamma(t),\, t\geq0,$}
\end{align}
where  $V(t)$ is the normal velocity of $\G(t)$ and $H(t)=k_1(t)+k_2(t)$ the mean curvature of $\Gamma(t)$, with $k_i(t),$ $i=1,\,2,$ denoting the principle curvatures of $\G(t)$,
we obtain the following evolution equation for the unknown function $u$:
\begin{subequations}\label{Pu}
\begin{equation}\label{Pu1}
u_t= \cfrac{u_{xx}}{1+u_x^2}-\cfrac{1}{u},\quad t\geq0, \, x\in(a(t),b(t)).
\end{equation}
We assumed that
\begin{equation}\label{Pu2}
u(t,x) >0,\quad   t\geq0,\, x\in(a(t),b(t)).
\end{equation} 
This equation cannot be realized at $x\in\{a(t),\, b(t)\}$ as we impose the following boundary conditions
\begin{align}\label{Pu3}
\left\{
\begin{array}{lllll}
 u(t,a(t)) =u(t,b(t))=0,\quad   t\geq0,\\[2ex]
  \underset{x\searrow a(t)}\lim u_x(t,x)=\infty,\,   \underset{x\nearrow b(t)}\lim u_x(t,x) = -\infty,\quad    t\geq0, 
\end{array}
\right.
\end{align}
which express the fact that $\Gamma(t)$  is a closed surface without boundary.\medskip

\noindent{\bf The evolution of the boundaries: The first approach.} As the functions $a$ and $b$ are unknown, we 
have to derive equations describing the evolution of these two boundaries.  
If we want to  evaluate the  normal velocity
at $(a(t),0,0)\in\Gamma(t)$, it follows from \eqref{Pu3} that $V(t)|_{(a(t),0,0)}=a'(t).$ 
The next goal is to express $H(t)$ at $(a(t),0,0)$ in terms of $u(t).$
To this end we assume, for some $\e>0$, that $$u(t):[a(t),a(t)+\e]\to[0,u(t, a(t)+\e)]$$ is invertible  with the inverse function $w(t)\in {\rm C}^2([0,u(t, a(t)+\e)]) $, 
so that in particular $\Gamma(t)$ is a ${\rm C}^2$-surface close to $(a(t),0,0)$.
Then due to the fact that $ u_x (t,x) \to\infty$ for $x\to a(t)$ we have
 $ w_y(t,0)=0$ and 
\begin{align*}
\lim_{x\to a(t)}k_1(t,x,\vartheta)=\lim_{x\to a(t)}\Big(-\frac{u_{xx}}{(1+u_x^2)^{3/2}}\Big)(t,x)=\lim_{y\to 0}\frac{w_{yy}}{(1+w_y^2)^{3/2}}(t,y)=w_{yy}(t,0),\\[1ex]
\lim_{x\to a(t)}k_2(t,x,\vartheta)=\lim_{x\to a(t)} \frac{1}{u(1+u_x^2)^{1/2}}(t,x)=\lim_{y\to 0}\frac{w_{y}}{y(1+w_y^2)^{1/2}}(t,y)=w_{yy}(t,0), 
\end{align*}
hence $H(t)|_{(a(t),0,0)}=2w_{yy}(t,0).$
Noticing that  
\[
\lim_{x\to a(t)}(uu_x)(t,x)=\lim_{y\to 0}\frac{y}{w_y(t,y)}=\frac{1}{w_{yy}(t,0)},
\]
in the case when  $H(t)|_{(a(t),0,0)}=2w_{yy}(t,0)>0$  we obtain the following relation
\begin{align}\label{Pu4a}
a'(t)=H(t)|_{(a(t),0,0)}=\frac{2}{\underset{x\to a(t)}\lim(u^2/2)_x(t,x)}.
\end{align}
Similarly, assuming that, for some $\e>0$,  $$u(t):[b(t)-\e,b(t)]\to[0,u(t, b(t)-\e)]$$ is invertible and $H(t)|_{(b(t),0,0)}>0$, we find for $b$ the evolution equation
\begin{align}\label{Pu4b}
b'(t)=H(t)|_{(b(t),0,0)}=\frac{2}{\underset{x\to b(t)}\lim(u^2/2)_x(t,x)}.
\end{align}
\end{subequations}
\noindent{\bf The evolution of the boundaries: The second approach.}
A major drawback of the (formally) quasilinear parabolic  equation \eqref{Pu1}  is  that the boundary conditions \eqref{Pu3}  make the equation highly degenerate as:
\begin{itemize}
\item[(I)] The diffusion coefficient vanishes in the limit $x\to a(t)$ and $x\to b(t)$;
\item[(II)] The term $1/u$ becomes unbounded for $x\to a(t)$  and $x\to b(t)$.
\end{itemize}
In order to overcome (II) we introduce, motivated also by \eqref{Pu4a}-\eqref{Pu4b}, a new unknown $v$ via
 \[
v(t,x):=\frac{u^2(t,x)}{2}\qquad\text{for $t\geq 0$ and $x\in[a(t),b(t)]$.} 
 \] 
 Then $v(t)$ also vanishes at the boundary points $a(t)$ and $b(t)$ and \eqref{Pu1} can be expressed as
 \begin{align}\label{SoMi}
v_t= \cfrac{2vv_{xx}}{2v+v_x^2}-\cfrac{v_x^2}{2v+v_x^2}-1,\qquad t\geq0, \, x\in(a(t),b(t)).
 \end{align}
This equation  is also (formally) quasilinear  parabolic and also degenerate -- as the  diffusion coefficient 
  vanishes for $x\to a(t)$ and $x\to b(t)$, cf. \eqref{SoRe} below -- but now none of the  coefficients is singular.
In order to obtain an evolution equation also for the functions describing the boundaries, we assume that  
\begin{align}\label{SoRe}
 v_x(t, a(t))>0 \quad\text{and}\quad v_x(t,b(t))<0,\qquad t\geq0,
\end{align} 
and 
\begin{align}\label{SoLa}
\lim_{x\to a(t)}(vv_{xx})(t,x)=\lim_{x\to b(t)}(vv_{xx})(t,x)=0,\qquad t\geq0.
\end{align}
 Note that \eqref{SoRe}  implies in particular that the corresponding function $u$ satisfies $\eqref{Pu3}.$
 Furthermore, \eqref{SoLa} is a nonlinear boundary condition for $v$ which is equivalent to our former assumption  that $\Gamma(t)$ is a ${\rm C}^2$-surface, cf. Lemma \ref{L:A-1}.
Differentiating now the relation  $v(t,a(t))=0$, $t\geq0$,  with respect
to time, it follows in virtue of \eqref{SoMi}  and \eqref{SoLa}, that
 \[
a'(t)= \frac{2}{v_x(t,a(t))} \quad\text{and}\quad  b'(t)= \frac{2}{v_x(t,b(t))} ,\qquad t\geq0.
 \]
These are the very same relations as in \eqref{Pu4a}-\eqref{Pu4b}.
It is not difficult to see, cf. Lemma \ref{L:A-1}, that the two approaches are equivalent.

Summarizing, we may formulate the problem by using $v$ as an unknown and we   arrive at the    evolution problem
 \begin{subequations}\label{Pv}
 \begin{align}\label{Pv1}
\left\{
\begin{array}{rlllll}
 &v_t= \cfrac{2vv_{xx}}{2v+v_x^2}-\cfrac{v_x^2}{2v+v_x^2}-1,\quad t\geq0, \, x\in(a(t),b(t)),\\[2ex]
 &a'(t) =\cfrac{2}{v_x(t,a(t))},\quad t\geq0,\\[2ex]
 &b'(t) =\cfrac{2}{v_x(t,b(t))},\quad t\geq0,\\[2ex]
 &v(t,a(t))=v(t,b(t))=0,\quad t\geq0,\\[2ex]
 &\underset{x\to a(t)}\lim(vv_{xx})(t,x)=\underset{x\to b(t)}\lim(vv_{xx})(t,x)=0,\quad t\geq0,\\[2ex]
 &v(t,x) >0,\quad   t\geq0,\, x\in(a(t),b(t)),\\[2ex]
 & v_x(t,a(t))>0,\,   v_x(t,b(t)) <0,\quad    t\geq0, 
\end{array}
\right.
\end{align}
with  initial conditions
  \begin{align}\label{Pv2}
v(0)=v_0,\quad a(0)=a_0,\quad b(0)=b_0. 
\end{align}
 \end{subequations}
In the following we use the formulation \eqref{Pv} in order to investigate the mean curvature flow \eqref{MCF}. We are interested here to prove the existence and uniqueness of solutions   
which satisfy the equations in  a classical sense (a weak formulation of \eqref{Pv} is not available yet).  
 The formulation \eqref{Pv} has two advantages compared to the
classical approach followed in  \cite{ESim98,Ma11,  DGK14} for example. Firstly, the equations are explicit (we do not need to work with local charts) and, secondly 
because the maximal solutions to \eqref{Pv} are defined in general on a larger time interval compared to the ones in
\cite{ESim98, Ma11,  DGK14} (the solutions in  \cite{ESim98, Ma11, DGK14} exist only  in a
small neighborhood  of a fixed reference manifold over which they are parameterized). 
A disadvantage of our approach is  that herein the 
initially surfaces are necessary of class ${\rm C}^2$ while in   \cite{ESim98}    only  ${\rm h}^{1+\alpha}$-regularity, for some fixed $\alpha\in(0,1)$,  is required.
  An interesting research topic which we next plan to follow is to determine initial data  $u_0:[a_0,b_0]\to\R$ for which the closed rotationally symmetric
   surface $\Gamma(0)$ evolves such that neck pinching at the origin occurs in finite time. This topic has been already studied in the context of \eqref{Pu1}, but in the special setting when $a(t)$ and $b(t)$ are kept fixed, the function $u$ is strictly  positive, and suitable boundary conditions (either of Neumann or Dirichlet type) are imposed at these two fixed boundary points, cf. \cite{MB07, EM10, DK91, Hu90, Mc15}. In the context of closed surfaces without boundary considered herein
   there are several results establishing the convergence  of initially convex surfaces towards a round point in finite time, cf. e.g. \cite{Mc15, AnM12, Hu84},
    but to the best of our knowledge no result establishing neck pinching at the origin is available. 
  It is worth to emphasize that in this context however, by using maximum principles and some explicit solutions to the mean curvature flow, such as spheres, hyperboloids, 
  or  shrinking donuts, there are several examples of dumbbell shaped  surfaces which   develop singularities in finite time, cf.  \cite{An92, GR89, Ec04}.

 \begin{rem} 
 \begin{itemize}
 \item[(i)] If $u(t,x)=\sqrt{a^2(t)-x^2}$, $|x|\leq a(t)$, then the surfaces under consideration are spheres and the radius $a(t)>0$ solves the ODE
 \[
a'(t)=-\frac{2}{a(t)}=\frac{2}{v_x(t,a(t))},\qquad t\geq0. 
 \]
 \item[(ii)] The conditions \eqref{SoRe}-\eqref{SoLa}  impose   some restrictions on the initial data. 
 Lemma \ref{L:A-1} shows that any rotationally symmetric surface of class ${\rm C}^2$  with  mean curvature that does not vanish  
 at the points on the rotation axis satisfies \eqref{SoRe}-\eqref{SoLa}.
 These properties are then  preserved by the flow.
 \end{itemize}

 \end{rem}

We will solve the degenerate parabolic system in the setting of small H\"older spaces.
The small H\"older space  ${\rm h}^{k+\alpha}(\s)$, $k\in\N$, $\alpha\in(0,1),$ is defined
as the closure of the smooth periodic functions ${\rm C}^\infty(\s)$
(or equivalently of ${\rm C}^{k+\alpha'}(\s)$, $\alpha'>\alpha$)
in the classical H\"older space  ${\rm C}^{k+\alpha}(\s)$
of $2\pi$-periodic functions on the line with $\alpha$-H\"older
continuous $k$-th derivatives.
Besides, ${\rm h}_e^{k+\alpha}(\s)$, denotes the subspace of ${\rm
h}^{k+\alpha}(\s)$ consisting only of even functions.
By definition, the embedding ${\rm h}_e^{r}(\s)\hookrightarrow {\rm
h}_e^{s}(\s)$, $r>s$, is dense  and
moreover it holds
  \begin{equation}\label{Intprop}
({\rm h}^{r}_e(\s),{\rm h}^{s}_e(\s))_\theta={\rm h}^{(1-\theta)r+\theta
s}_e(\s)\qquad\text{for $ \theta\in(0,1)$ and $(1-\theta)r+\theta
s\notin\N.$}
\end{equation}
Here $(\cdot,\cdot)_\theta=(\cdot,\cdot)_{\theta,\infty}^0$ denotes the
continuous interpolation functor introduced by Da Prato and Grisvard \cite{DG79}.
\pagebreak
 
 The main result of this paper is the following theorem.
 
 \begin{thm}\label{MT1} Let $\alpha\in(0,1)$ be a fixed H\"older exponent,  $a_0< b_0\in\R$, 
 and  let $v_0\in  {\rm C}^1([a_0,b_0]) $ be positive in $(a_0,b_0)$ such that $v_0(a_0)=v_0(b_0)=0,$ $v_0'(a)>0>v_0'(b),$ and 
  \[
 v_0((a_0+b_0)/2-(b_0-a_0)\cos(\cdot)/2) \in {\rm h}_e^{2+\alpha}(\s). 
 \]
 Then the evolution problem \eqref{Pv} has a unique maximal solution $(v,a,b):=(v,a,b)(\,\cdot\,;(v_0,a_0,b_0))$ such that 
 \begin{align*}
&h\in  {\rm C}^1([0,t^+),{\rm h}^{\alpha}_e(\s))\cap {\rm C}([0,t^+),{\rm h}^{2+\alpha}_e(\s)),\\[1ex]
&a,\, b\in {\rm C}^1([0,t^+),\R) ,\\[1ex]
&a(t)<b(t)\quad\text{for all $t\in[0,t^+)$},
 \end{align*}
 where   
 $$h(t,x):=v(t,(a(t)+b(t))/2-(b(t) -a(t) )\cos(x)/2),\quad t\in[0,t^+),\, x\in\R,$$ 
 and $t^+:=t^+(v_0,a_0,b_0)\in(0, \infty]$. 
 Moreover, it holds that 
 \[
a,\, b\in{\rm C}^\omega((0,t^+)),\qquad v\in {\rm C}^\omega(\{(t,x)\,:\, 0<t<t^+,\, a(t)<x<b(t)\},(0,\infty)). 
 \] 
 \end{thm}

 \begin{rem}\label{R:21}
 \begin{itemize}
 \item[(i)] The choice of the small H\"older spaces is essential. Indeed, using  a singular transformation from \cite{An88},
  we may recast the evolution problem \eqref{Pv}
 as a fully nonlinear  evolution equation with the leading order term in $\eqref{Pv1}_1$ having in the linearisation - when working within this class of functions - a positive and bounded coefficient. Besides, the setting of small H\"older spaces is a smart  choice when dealing with fully nonlinear parabolic equations, cf. e.g. \cite{DG79, L95}. 
 A further departure of these spaces from the classical H\"older spaces is illustrated in Lemma \ref{L:A1}.
 \item[(ii)] The problem considered in \cite{An88} is general enough to include also \eqref{Pv}. 
 However, the technical details, see Section \ref{Sec:3}, are different from those in  \cite{An88} and also simpler.  
  Besides, the parabolic smoothing property for $v$ in Theorem \ref{MT1} is a new result in this degenerate parabolic  setting and it extends also to the general problem
    considered in \cite{An88}.
  In particular, this proves in the context of  the porous medium equation (which is the equation that motivates the analysis  in \cite{An88})
   that the interface separating a fluid blob, that expends under the effect of gravity,  from air is real-analytic in the positivity set, see \cite{Va07} for more references on this topic.  
  
  \item[(iii)] If $v_0\in {\rm h}^{2+\alpha}([a_0,b_0])$  satisfies $v_0(a_0)=v_0(b_0),$ $v_0>0$ in $(a_0,b_0)$, and  
 \[ v_0'(a_0)>0> v_0'(b_0),\]
 then $v_0$ can be chosen as an initial condition in \eqref{Pv}. However,
 the initial data $v_0$ in Theorem \ref{MT1} are not required to be twice differentiable at $x=a_0$ and $x=b_0$. For example if $b_0=-a_0=1$, then 
 \[
v_0(x):= 1-x^2+(1-x^2)^{3/2} , \qquad x\in[-1,1],
 \] 
 can be chosen as an initial condition in \eqref{Pv}  as $v_0\circ (-\cos)\in W^3_\infty(\s)$, but $v_0\not\in {\rm C}^2([-1,1])$.
 \item[(iv)] The assumption  that $h_0:= v_0((a_0+b_0)/2-(b_0-a_0)\cos/2) \in {\rm h}_e^{2+\alpha}(\s)$ guarantees that the nonlinear boundary condition $\eqref{Pv1}_5$
 holds at $t=0$. Indeed,    since 
 \[
(v_0v_0'')\circ((a_0+b_0)/2-(b_0-a_0)\cos/2)= \frac{4}{(b-a)^2}\frac{h_0}{\sin^2}\Big(h_0''-\frac{h_0'}{\tan}\Big),
 \]
l'Hospital's rule shows  that   $\underset{x\to a_0}\lim(v_0v_{0}'')(x)=\underset{x\to b_0}\lim(v_0v_{0}'')(x)=0$.
 \end{itemize}
 \end{rem}
 
  %%%%%%%%%%%%%%%%%%%%%%%%%%%%%%%%%%%%%%%%%%%%%%%%%%%%%%%%%%%%%%%%%%%
%%%%%%%%%%%%%%%%%%%%%%%%%%%%%%%%%%%%%%%%%%%%%%%%%%%%%%%%%%%%%%%%%%%%
%%%%%%%%%%%%%%%%%%%%%%%%%%%%%%%%%%%%%%%%%%%%%%%%%%%%%%%%%%%%%%%%%%%%
%%%%%%%%%%%%%%%%%%%%%%%%%%%%%%%%%%%%%%%%%%%%%%%%%%%%%%%%%%%%%%%%%%%
%%%%%%%%%%%%%%%%%%%%%%%%%%%%%%%%%%%%%%%%%%%%%%%%%%%%%%%%%%%%%%%%%%%%
%%%%%%%%%%%%%%%%%%%%%%%%%%%%%%%%%%%%%%%%%%%%%%%%%%%%%%%%%%%%%%%%%%%%
 
\section{The transformed problem}\label{Sec:2}
 %%%%%%%%%%%%%%%%%%%%%%%%%%%%%%%%%%%%%%%%%%%%%%%%%%%%%%%%%%%%%%%%%%%
%%%%%%%%%%%%%%%%%%%%%%%%%%%%%%%%%%%%%%%%%%%%%%%%%%%%%%%%%%%%%%%%%%%%
%%%%%%%%%%%%%%%%%%%%%%%%%%%%%%%%%%%%%%%%%%%%%%%%%%%%%%%%%%%%%%%%%%%%
%%%%%%%%%%%%%%%%%%%%%%%%%%%%%%%%%%%%%%%%%%%%%%%%%%%%%%%%%%%%%%%%%%%
%%%%%%%%%%%%%%%%%%%%%%%%%%%%%%%%%%%%%%%%%%%%%%%%%%%%%%%%%%%%%%%%%%%%
%%%%%%%%%%%%%%%%%%%%%%%%%%%%%%%%%%%%%%%%%%%%%%%%%%%%%%%%%%%%%%%%%%%%

In order to study \eqref{Pv} we use an idea from  \cite{An88} and transform the evolution problem \eqref{Pv} into a system defined in the setting of periodic functions 
 by using a diffeomorphism that has a  
first derivative which is singular at the points $a(t) $ and $b(t)$.
More precisely, we introduce the new unknown 
$$h(t,x):=v(t,c(t)-d(t)\cos(x)),\qquad x\in\R,\, t\geq0,$$
where
\[
c(t):=\frac{a(t)+b(t)}{2}\quad\text{and}\quad d(t):=\frac{b(t)-a(t)}{2}>0.
\]
Given  $t\geq0$,   $h(t)$ is a $2\pi$-periodic function on $\R$ which is even and  merely the continuous  differentiability of $v(t)$   implies that
\[
h_{xx}(t,0)= \frac{b(t)-a(t)}{a'(t)}>0 \quad\text{and}\quad h_{xx}(t,\pi)= -\frac{b(t)-a(t)}{b'(t)}>0.
\] 
In terms of the new variable $(h,c,d)$ the problem \eqref{Pv} can be recast as follows
 \begin{subequations}\label{TP}
 \begin{equation}\label{TP1}
\left\{
\begin{array}{rlllll}
& h_t = \cfrac{2}{2d^2h+ h_x^2/\sin^2}\cfrac{h}{\sin^2}
 \Big(h_{xx}-  \cfrac{h_x}{\tan} \Big)
 -\cfrac{ h_x^2/\sin^2}{2d^2h+ h_x^2/\sin^2}-1\\[2ex]
&\hspace{1cm} + \Big(\cfrac{1+\cos}{h_{xx}(t,0)}-\cfrac{1-\cos}{h_{xx}(t,\pi)}\Big)\cfrac{h_x}{\sin}, \quad t\geq0, \, x\in\R,\\[2ex]
  &c' =d\Big(\cfrac{1}{h_{xx}(t,0)}-\cfrac{1}{h_{xx}(t,\pi)}\Big),\quad t\geq0,\\[2ex]
 &d' =-d\Big(\cfrac{1}{h_{xx}(t,0)}+\cfrac{1}{h_{xx}(t,\pi)}\Big),\quad t\geq0,\\[2ex]
&d(t)>0,\, h(t,0)=h(t,\pi)=0,\quad t\geq0,\\[2ex]
 & h(t,x) >0,\quad   t\geq0,\, x\in(0,\pi),\\[2ex]
  & h_{xx}(t,0)>0,\,   h_{xx}(t,\pi) >0,\quad    t\geq0, 
\end{array}
\right.
\end{equation}
 with initial conditions
  \begin{align}\label{PP2}
h(0)=h_0:=v_0(c_0-d_0\cos),\qquad c(0)=\frac{a_0+b_0}{2},\qquad d(0)=\frac{b_0-a_0}{2}. 
\end{align}
 \end{subequations}
We point out that nonlinear boundary condition $\eqref{Pv1}_5$ has not
been  taken into account in the transformed system \eqref{TP}. 
This is due to the choice of the function spaces below as, similarly as in Remark~\ref{R:21}~(iv), requiring that $h(t)\in\E_1$ ensures that $\eqref{Pv1}_5$
holds at time $t\geq0$.

In order to study  \eqref{TP}  we choose as an appropriate  framework  the setting of periodic small H\"older spaces. 
 For a  fixed $\alpha\in(0,1)$ we  define the Banach spaces
\begin{align*}
& \E_0:=\{h\in {\rm h}^\alpha_e(\s)\,:\, \text{ $h(0)=h(\pi)=0$}\},\\[1ex]
& \E_1:=\{h\in {\rm h}^{2+\alpha}_ e(\s)\,:\, \text{ $h(0)=h(\pi)=0$}\},
\end{align*}
with the corresponding norms $\|\cdot\|_i=\|\cdot\|_{ {\rm C}^{2i+\alpha}(\s)},$ $i\in\{0,1\}.$
It is important to point out that the embedding $\E_1\hookrightarrow\E_0$ is dense.
Though at formal level the equation $\eqref{TP1}_1$ has a  quasilinear structure, our analysis below shows that the problem \eqref{TP}   is actually (as a result of the boundary conditions)
 fully nonlinear (see Lemma \ref{L:1} and the subsequent discussion).
This loss of linearity is however compensated by the fact  that none of  the terms on the  right hand side of $\eqref{TP1}_1$ is singular when choosing $h\in \E_1$. 
Moreover the  function multiplying $h_{xx}$ in $\eqref{TP1}_1$ is now $\alpha$-H\"older continuous and positive.

\begin{lemma}\label{L:1}
The operators 
\[
\Big[h\mapsto \frac{h}{\sin^2}\Big],\, \Big[h\mapsto \frac{ h'}{\sin}\Big]:\E_1\to {\rm h}^{\alpha}_e(\s)
\]
are bounded.
\end{lemma}
\begin{proof}
 See  \cite[Lemma 2.1]{An88}.
\end{proof}
We  emphasize that  it is not possible to choose in Lemma \ref{L:1} as  target space a small H\"older space ${\rm h}^{\alpha'}_e(\s) $
 with $\alpha'>\alpha$.
In particular, the terms $h/\sin^2$ and $h_x/\tan$ on the right-hand side of $\eqref{TP1}_1$ have the same importance as $h_{xx}$ when linearizing this expression.

We now set 
\[
\cO:=\{h\in \E_1\,:\, \text{$h''(0)>0,$ $h''(\pi)>0$ and $h>0$ in $(0,\pi)$}\}.
\] 
Then, $\cO$ is an open subset of $\E_1$.
Let further
\begin{align*}
\Phi:=(\Phi_1,\Phi_2, \Phi_3):\cO\times\R\times (0,\infty)\subset \E_1\times \R^2\to\E_0\times \R^2
\end{align*}
be the operator defined by 
\begin{align*} 
&\Phi_1(h,c,d):= \cfrac{2}{2d^2h+ h'^2/\sin^2}\cfrac{h}{\sin^2}
 \Big(h''- \cfrac{h'}{\tan} \Big)
 -\cfrac{ h'^2/\sin^2}{2d^2h+h'^2/\sin^2}-1+\Big(\cfrac{1+\cos}{h''(0)}-\cfrac{1-\cos}{h''(\pi)}\Big)\cfrac{h'}{\sin},\\[1ex]
&\Phi_2(h,c,d):=d\Big(\cfrac{1}{h''(0)}-\cfrac{1}{h''(\pi)}\Big),\\[1ex]
&\Phi_3(h,c,d):=-d\Big(\cfrac{1}{h''(0)}+\cfrac{1}{h''(\pi)}\Big).
\end{align*}
It is not difficult to check  that $\Phi_1(h,c,d)|_{x=0}=\Phi_1(h,c,d)|_{x=\pi}=0,$ so that  $\Phi$ is well-defined. In virtue of Lemma \ref{L:1} it further holds that 
\begin{align}\label{REG}
\Phi\in {\rm C}^\omega(\cO\times\R\times (0,\infty) ,\E_0\times\R^2).
\end{align}
Hence, we are led to the fully nonlinear evolution problem
\begin{align}\label{FNP}
(\dot h,\dot c,\dot d)=\Phi(h,c,d),\, \, t\geq0,\qquad (h(0),c(0), d(0))=(h_0,c_0,d_0),
\end{align}
with $(h_0,c_0,d_0)\in\cO\times\R\times (0,\infty)$.
We shall establish the existence and uniqueness of strict solutions (in the sense of  \cite{L95}) to \eqref{FNP}  by using the fully nonlinear parabolic 
theory presented in the monograph \cite{L95}.
To this end we next   identify the Fr\'echet derivative $\p\Phi(h_0,c_0,d_0)$ and we prove that it generates, for each $(h_0,c_0,d_0)\in\cO\times\R\times (0,\infty)$, a 
strongly continuous and analytic semigroup. In the notation of
Amann \cite{Am95} this means by definition
\[
-\p\Phi(h_0,c_0,d_0)\in\kH(\E_1\times\R^2,\E_0\times\R^2).
\]
In fact, in view of \cite[Corollary I.1.6.3]{Am95}, we only need
to show that   the partial derivative
$\p_h\Phi_1(h_0,c_0,d_0)$ generates a strongly
continuous analytic semigroup  in $\kL(\E_0).$
 Given $h\in\E_1$, it holds that 
 \begin{align*}
 \p_h\Phi_1(h_0,c_0,d_0)[h]&= A_1\Big(h''-\cfrac{h'}{\tan} \Big) 
 +A_2 \cfrac{h'}{\sin}+A_3\frac{h}{\sin^2}+A_4h +A_5h''(0)+A_6h''(\pi)
 \end{align*}
where
\begin{align*}
A_1&:= \cfrac{2}{2d^2_0h_0+ h'^2_0/\sin^2}\cfrac{h_0}{\sin^2},\\[1ex]
A_2&:=  \cfrac{2 }{ \big(2d^2_0h_0+ h'^2_0/\sin^2\big)^2}
\Big[\cfrac{h_0'^2}{\sin^2} - \cfrac{2h_0}{\sin^2}\Big( h_0''-\cfrac{h_0'}{\tan} \Big) \Big]
  \cfrac{h_0'}{\sin }
  - \cfrac{2 }{ 2d^2_0h_0+ h'^2_0/\sin^2  }
\cfrac{h_0'}{\sin}        + \cfrac{1+\cos}{h_0''(0)}-\cfrac{1-\cos}{h_0''(\pi)},   \\[1ex]          
  A_3& :=\cfrac{2}{2d^2_0h_0+ h'^2_0/\sin^2}\Big( h_0''-\cfrac{h_0'}{\tan} \Big),\quad A_4:=  \cfrac{2d_0^2 }{ \big(2d^2_0h_0+ h'^2_0/\sin^2\big)^2}
\Big[\cfrac{h_0'^2}{\sin^2} - \cfrac{2h_0}{\sin^2}\Big( h_0''-\cfrac{h_0'}{\tan} \Big) \Big],\\[1ex]
A_5&:=-\cfrac{h'_0}{\sin}\frac{1+\cos}{(h_0''(0))^2},\quad A_6:=\cfrac{h'_0}{\sin}\frac{1-\cos}{(h_0''(\pi))^2}.
\end{align*}
We note that $A_i\in {\rm h}^\alpha_e(\s)$, $1\leq i\leq 6$, with $A_1$ being positive. Moreover, it holds that $A_3\in\E_0$.

Since
\[
\|A_4 h\|_0\leq \|A_4\|_0\|h\|_{0} \qquad\text{for all $h\in\E_1$},
\]
the operator $[h\mapsto A_4h]:\E_1\to\E_0 $ may be viewed as being a lower order perturbation of  $\p_h\Phi_1(h_0,c_0,d_0)$,  cf. \cite[Theorem I.1.3.1 (ii)]{Am95}.
The following result enables us to regard also other terms of $\p_h\Phi_1(h_0,c_0,d_0)$   as being lower order perturbations.
\begin{lemma}\label{L22}
 Let $A\in \E_0$.
 Then, given $\e>0$, there exists a constant $C(\e)>0$ such that
 \begin{align}\label{Pert}
 \Big\|A \frac{h}{\sin^2}\Big\|_0+\Big\|A \frac{h'}{\sin}\Big\|_0\leq \e\|h\|_1+C(\e)\|h\|_0
 \qquad\text{for all $h\in\E_1.$}
 \end{align}
\end{lemma}
\begin{proof}
Letting $I:=[-2\pi/3,5\pi/3]$, it is not difficult to verify that
\begin{align}\label{EQN}
\|h\|_{{\rm C}^{k+\alpha }(I)}\leq \|h\|_{{\rm C}^{k+\alpha }(\s)}\leq 5\|h\|_{{\rm C}^{k+\alpha}(I)},\qquad k\in\N,\,\alpha\in(0,1),\, h\in {\rm C}^{k+\alpha}(\s).
\end{align}
 In view of this equivalence, the claim for $[h\mapsto A   h'/\sin ]$ follows from the observation that
 \begin{align*}
 &\frac{ h'(x)}{\sin(x)}=\varphi_1(x)\int_{0}^1h''(sx)\, ds,\quad x\in[-2\pi/3,\pi/2],\\[1ex]
  &\frac{ h'(x)}{\sin(x)}=\varphi_2(x)\int_{0}^1h''((1-s)\pi+sx)\, ds,\quad x\in[\pi/2,5\pi/3],
 \end{align*}
 where the functions $\varphi_1(x):=x/\sin(x)$ and $\varphi_2(x):=(x-\pi)/\sin(x)$  belong to ${\rm C}^\infty([-2\pi/3,\pi/2]) $ and ${\rm C}^\infty([\pi/2,5\pi/3]),$ respectively.
 The proof of the second claim follows by similar arguments.
\end{proof}

Recalling that $A_3\in\E_0$,  Lemma \ref{L22} implies that also  $[h\mapsto A_3 h/\sin^2]:\E_1\to\E_0$ can be viewed as being a perturbation.
Let us now notice that 
$$A_2(0)=2/h_0''(0)=2A_1(0)\quad\text{and}\quad A_2(\pi)=-2/h_0''(\pi)=-2A_1(\pi).$$
Observing that
\begin{align*}
A_2 \cfrac{h'}{\sin}=(A_2-2A_1\cos)\cfrac{h'}{\sin}+2A_1\cfrac{h'}{\tan},
\end{align*}
where $A_2-2A_1\cos\in\E_0$, we may regard in view of  Lemma \ref{L22} also the operator  
$$[h\mapsto(A_2-2A_1\cos) h'/\sin]:\E_1\to\E_0$$
 as being a perturbation
 and we are left  to prove the generator property for  
\[
\wt\bA:=\Big[h\mapsto A_1\Big(h''+\cfrac{h'}{\tan} \Big) +A_5h''(0)+A_6h''(\pi)\Big]:\E_1\to\E_0
.
\]
In fact, it suffices to establish  the generator property
for the operator
\begin{align}\label{bA}
\bA :=\Big[h\mapsto A_1\Big(h''+\cfrac{h'}{\tan} \Big)\Big]:{\rm h}_e^{2+\alpha}(\s)\to{\rm h}_e^{\alpha}(\s)
\end{align}
where we have dropped the lower order term $[h\mapsto A_5h''(0)+ A_6h''(\pi)].$
Indeed, assuming that $-\bA\in\kH({\rm h}_e^{2+\alpha}(\s),{\rm h}_e^{\alpha}(\s))$, it follows  $-\wt\bA\in\kH({\rm h}_e^{2+\alpha}(\s),{\rm h}_e^{\alpha}(\s))$.
This latter property is equivalent to the existence of constants $\kappa\geq 1$ and $\omega>0$ such that
\begin{align*}
&\text{(1) \quad $\omega-\wt\bA:{\rm h}_e^{2+\alpha}(\s)\to {\rm h}_e^{\alpha}(\s)$ is an isomorphism, and}\\[1ex]
&\text{(2) \quad $ \kappa^{-1}\leq\cfrac{\|(\lambda-\wt\bA)[h]\|_0}{|\lambda|\cdot\|h\|_0+\|h\|_1}\leq\kappa$ 
for all $\re\lambda\geq\omega$ and $0\neq h\in {\rm h}_e^{2+\alpha}(\s)$,} 
\end{align*}
 cf. \cite[Chapter I]{Am95}. The relation $(2)$ holds in particular for $0\neq h\in\E_1$. 
In order to conclude that $-\wt \bA\in\kH(\E_1,\E_0),$ we are thus left to show that $\omega-\wt\bA:\E_1\to\E_0$ is an isomorphism too.
Hence, given $f\in\E_0$, for $h\in{\rm h}_e^{2+\alpha}(\s) $ with $(\omega- \wt\bA)[h]=f $ we set
\[
\wt h:=\frac{1+\cos}{2}h(0)+\frac{1-\cos}{2}h(\pi).
\]
Taking into account that  $h-\wt h\in\E_1,$ it follows that $(\omega-\wt \bA)[\wt h]\in \E_0.$ 
A simple computation shows that $ \wt \bA [\wt h]\in\E_0$, so that also $\wt h\in\E_0$. We may thus  conclude that $h\in\E_1$, so that $(1)$ holds also when replacing 
${\rm h}_e^{2i+\alpha}(\s)$ with $\E_i$, $i\in\{0,\,1\}$.
The nontrivial property $-\bA\in\kH({\rm h}_e^{2+\alpha}(\s),{\rm h}_e^{\alpha}(\s))$ is  established in detail in Section \ref{Sec:3} below, cf. Theorem \ref{T:31}.

  %%%%%%%%%%%%%%%%%%%%%%%%%%%%%%%%%%%%%%%%%%%%%%%%%%%%%%%%%%%%%%%%%%%
%%%%%%%%%%%%%%%%%%%%%%%%%%%%%%%%%%%%%%%%%%%%%%%%%%%%%%%%%%%%%%%%%%%%
%%%%%%%%%%%%%%%%%%%%%%%%%%%%%%%%%%%%%%%%%%%%%%%%%%%%%%%%%%%%%%%%%%%%
%%%%%%%%%%%%%%%%%%%%%%%%%%%%%%%%%%%%%%%%%%%%%%%%%%%%%%%%%%%%%%%%%%%
%%%%%%%%%%%%%%%%%%%%%%%%%%%%%%%%%%%%%%%%%%%%%%%%%%%%%%%%%%%%%%%%%%%%
%%%%%%%%%%%%%%%%%%%%%%%%%%%%%%%%%%%%%%%%%%%%%%%%%%%%%%%%%%%%%%%%%%%%
 
\section{The generator property}\label{Sec:3}
 %%%%%%%%%%%%%%%%%%%%%%%%%%%%%%%%%%%%%%%%%%%%%%%%%%%%%%%%%%%%%%%%%%%
%%%%%%%%%%%%%%%%%%%%%%%%%%%%%%%%%%%%%%%%%%%%%%%%%%%%%%%%%%%%%%%%%%%%
%%%%%%%%%%%%%%%%%%%%%%%%%%%%%%%%%%%%%%%%%%%%%%%%%%%%%%%%%%%%%%%%%%%%
%%%%%%%%%%%%%%%%%%%%%%%%%%%%%%%%%%%%%%%%%%%%%%%%%%%%%%%%%%%%%%%%%%%
%%%%%%%%%%%%%%%%%%%%%%%%%%%%%%%%%%%%%%%%%%%%%%%%%%%%%%%%%%%%%%%%%%%%
%%%%%%%%%%%%%%%%%%%%%%%%%%%%%%%%%%%%%%%%%%%%%%%%%%%%%%%%%%%%%%%%%%%%
The  first goal of this section is to establish Theorem \ref{T:31}, which is a main ingredient in the proof of the  main result.

\begin{thm}\label{T:31}
Given $h_0\in\cO$, it holds that $-\bA\in\kH({\rm h}_e^{2+\alpha}(\s),{\rm h}_e^{\alpha}(\s))$.
\end{thm}

We  consider for $\e\in(0,\e_0]$, with $\e_0>0$  sufficiently small, partitions $\{\pi_1^\e,\, \pi_2^\e, \pi_3^\e\}\subset {\rm C}^{ \infty}(I,[0,1])$ of the interval $I=[-2\pi/3,5\pi/3]$ and corresponding  
 families $\{\chi_1^\e,\, \chi_2^\e,\, \chi_2^\e\}\subset {\rm C}^{ \infty}(I,[0,1])$ with the following properties
\begin{itemize}
\item $\pi_1^\e+\pi_2^\e+\pi_3^\e =1$ in ${\rm C}^{ \infty}(I)$; 
\item  ${\rm supp\,}(\pi_1^\e)= [-3\e,3\e]$, ${\rm supp\,}(\pi_2^\e)= [\pi-3\e,\pi+3\e]$,  ${\rm supp\,}(\pi_3^\e)=I\setminus\big(  [-2\e,2\e]\cup[\pi-2\e,\pi+2\e])$;
\item $\chi_i^\e=1$ on ${\rm supp\,}(\pi_i^\e)$, $1\leq  i\leq 3$;
\item ${\rm supp\,}(\chi_1^\e)= [-4\e,4\e]$, ${\rm supp\,}(\chi_2^\e)= [\pi-4\e,\pi+4\e]$,  ${\rm supp\,}(\chi_3^\e)=I\setminus\big(  [-\e,\e]\cup[\pi-\e,\pi+\e])$;
\item  $\pi_1^\e $ and $\pi_2^\e(\pi+\cdot)$ are even on $[-3\e,3\e]$;
 \item $\pi_3^\e $ has an even and periodic extension in ${\rm C}^\infty(\s)$.
 \end{itemize}
 Extending  $\pi_1^\e $ and $\pi_2^\e(\pi+\cdot)$  by zero in $\R\setminus[-3\e,3\e]$, we may view these functions as being smooth and even functions on $\R$.

As a first step towards proving Theorem \ref{T:31} 
we   approximate $\bA$ locally by certain operators  which  are simpler to analyze. 
\begin{lemma}\label{L31} Let $\mu>0$   be given. 
Then, there exists $\e>0$, a constant $K=K(\e)>0$, and a partition 
$\{\pi_1^\e,\, \pi_2^\e,\, \pi_3^\e\}$ such that the operator $\bA$ introduced in \eqref{bA} satisfies
\begin{align}\label{DE}
\|\pi_i^\e\bA[h]-  \bA_{i}[\pi_i^\e h]\|_{{\rm C}^{ \alpha}(I)}\leq \mu\|\pi_i^\e h\|_{{\rm C}^{2+\alpha}(I)}+K\|h\|_{{\rm C}^{2 }(I)}
\end{align}
for $1\leq i\leq3$ and  $h\in {\rm h}_e^{2+\alpha}(\s)$, where
\begin{align}\label{OP1}
 \bA_1&= A_1(0) \Big(\p_x^2+\frac{1}{x}\p_x\Big),\\[1ex]
  \bA_2&= A_1(\pi) \Big(\p_x^2+\frac{1}{x-\pi}\p_x\Big),\label{OP2'}\\[1ex]
  \bA_3&=A_1\p_x^2.\label{OP2}
\end{align}
\end{lemma}
\begin{proof}
Observing that  $\pi_3^\e/\tan\in {\rm C}^{ \infty}(I)$, it follows that
\begin{align*} 
\|\pi_3^\e\bA[h]-  \bA_{3}[\pi_3^\e h]\|_{{{\rm C}^{ \alpha}(I)}}&\leq\|A_1[(\pi_3^\e)''h+2(\pi_3^\e)'h']\|_{{\rm C}^{ \alpha}(I)}
+\| A_1 h'\pi_3/\tan \|_{{\rm C}^{ \alpha}(I)}\leq K\|h\|_{{{\rm C}^{1+\alpha}(I)}},  
\end{align*}
which proves \eqref{DE} for $i=3$.

Furthermore, it holds that
\begin{align*} 
 \pi_1^\e\bA[h]-  \bA_{1}[\pi_1^\e h]=T_1[h]+T_2[h], 
\end{align*}
where
\begin{align*} 
T_1[h]&:=A_1 \pi_1^\e h''-A_1(0)(\pi_1^\e h)'',\\[1ex]
 T_2[h]&:=A_1 \frac{1}{\tan} \pi_1^\e h'-A_1(0)\frac{1}{x}(\pi_1^\e h)'.
\end{align*}
Using $\chi_1^\e\pi_1^\e=\pi_1^\e$, we now obtain
\begin{align*} 
\|T_1[h]\|_{{{\rm C}^{ \alpha}(I)}}&\leq \|(A_1-A_1(0))\chi_1^\e\|_{{{{\rm C} (I)}}} \|\pi_1^\e h\|_{{{{\rm C}^{ 2+\alpha}(I)}}}+K\| h\|_{{{\rm C}^{2}(I)}}\leq
\frac{\mu}{2}\|\pi_1^\e h\|_{ {{{\rm C}^{ 2+\alpha}(I)}} }+\|h\|_{{{{\rm C}^{ 2}(I)}}},
\end{align*}
provided that $\e$ is sufficiently small.

Concerning the second term we write
 \begin{align*} 
 T_2[h]=T_{2a}[h]+T_{2b}[h]-T_{2c}[h],
\end{align*}
where
 \begin{align*} 
 T_{2a}[h]& = (A_1-A_1(0))\chi_1^\e    \frac{1}{\tan} (\pi_1^\e h)',\\[1ex]
 T_{2b}[h]&=A_1(0)\Big(\frac{1}{\tan}-\frac{1}{x } \Big)(\pi_1^\e h)',\\[1ex]
 T_{2c}[h]&= A_1  \frac{1}{\tan}(\pi_1^\e )'h.
\end{align*}
The arguments in the proof of Lemma \ref{L22}   yield
 \begin{align*} 
\|T_{2a}[h]\|_{{{\rm C}^{ \alpha}(I)}}&\leq \|(A_1-A_1(0))\chi_1^\e\|_{{{\rm C}(I)}} \| (\pi_1^\e h)'/\tan\|_{{{\rm C}^{ \alpha}(I)}}
+K \|  (\pi_1^\e h)'/\tan\|_{{{\rm C} (I)}}
 \\
&\leq\frac{\mu}{2}\|\pi_1^\e h\|_{{{\rm C}^{2+ \alpha}(I)}}+K\|h\|_{{{\rm C}^{ 2}(I)}}.
\end{align*}
Besides, since $(\pi_1^\e)'/\tan\in {\rm C}^{ \infty}(I)$, we get
\[
\|T_{2c}[h]\|_{{{\rm C}^{ \alpha}(I)}}\leq K\|h\|_{{{\rm C}^{ \alpha}(I)}}\leq K\|h\|_{{{\rm C}^{2}(I)}}.
\]
Finally, it is not difficult to see that the function 
\[
\phi(x):=\frac{1}{\tan}-\frac{1}{x}
\]
satisfies  $ \chi_1^\e\phi\in{\rm C}^\infty(I)$. Therewith we have
 \begin{align*} 
\|T_{2b}[h]\|_{{{\rm C}^{ \alpha}(I)}}&\leq C\|\chi_1^\e\phi \|_{{{\rm C}^{ \alpha}(I)}}\| (\pi_1^\e h)'\|_{{{\rm C}^{ \alpha}(I)}} \leq K\|h\|_{{{\rm C}^{ 1+\alpha}(I)}},
\end{align*}
and we conclude that
 \begin{align*} 
\|T_{2}[h]\|_{{{\rm C}^{ \alpha}(I)}}\leq\frac{\mu}{2}\|\pi_1^\e h\|_{{{\rm C}^{2+ \alpha}(I)}}+K\|h\|_{{{\rm C}^{2}(I)}},
\end{align*}
provided that $\e$ is sufficiently small.
This proves \eqref{DE} for $i=1$. The proof of the claim for $i=2$ is similar and we therefore omit it.
\end{proof}

We now consider the operators $\bA_i$, $1\leq i\leq 3$, found in Lemma \ref{L31} in suitable functional analytic settings.
Regarding    $\bA_3$    as an element of  $\kL({\rm h}_e^{2+\alpha}(\s),{\rm h}_e^{\alpha}(\s))$, 
it is well-known that $\bA_3$ generates an analytic semigroup in $\kL({\rm h}_e^{\alpha}(\s))$. 
In particular, there exist  constants $\kappa_3\geq 1$ and $\omega_3>0$ such that 
\begin{align}\label{EDI1}
\kappa_3\|(\lambda-\bA_3)[h]\|_{0}\geq |\lambda|\cdot \|h\|_0+\|h\|_{1}, \qquad  h\in {\rm h}_e^{2+\alpha}(\s),\, \re\lambda\geq \omega_3,
\end{align}
cf. \cite[Theorem I.1.2.2]{Am95}.
The operator  $\bA_1$ can be viewed as an element of $\kL({\rm h}_{e}^{2+\alpha}(\R),{\rm h}_{e}^{\alpha}(\R))$ 
\footnote{For a definition of ${\rm h}^{k+\alpha}(\R^n)$, $k,\, n\in\N$, see \cite{L95}. Again, ${\rm h}_e^{k+\alpha}(\R)$, $k\in\N$, denotes 
the closed subspace of ${\rm h}^{k+\alpha}(\R)$ consisting of even functions.}.
Furthermore, in this context $\bA_1$ 
appears  as the restriction of $A_1(0)\Delta\in\kL({\rm h}^{2+\alpha}(\R^2), {\rm h}^{\alpha}(\R^2)) $ to the subset of rotationally symmetric functions.
Indeed, given $h\in {\rm h}_{e}^{k+\alpha}(\R),$  $k\in\{0,2\}$,  let
\[u(z):=(h\circ|\,\cdot\,|)(z)=h(\sqrt{x^2+y^2}),\qquad z=(x,y)\in\R^2.\]
One can show  that the radially symmetric function  $u$ belongs to $ {\rm h}^{k+\alpha}(\R^2)$ and that 
\begin{align*}
&\|h\|_{{\rm C}^{\alpha}(\R)}=\|u\|_{{\rm C}^{ \alpha}(\R^2)}, \\[1ex]
&\|h\|_{{\rm C}^{2+\alpha}(\R)}\leq \|u\|_{{\rm C}^{2+\alpha}(\R^2)}\leq C\|h\|_{{\rm C}^{2+\alpha}(\R^2)},
\end{align*}
with a constant $C\geq 1$ independent of $h$.
Recalling that $-A_1(0)\Delta\in\mathcal{H}({\rm h}^{2+\alpha}(\R^2), {\rm h}^{\alpha}(\R^2)) $, cf. \cite[Theorem 3.1.14 and Corollary 3.1.16]{L95},
there exist constants $\kappa_1\geq1$ and $\omega_1>0$ such that
\[
\kappa_1\|(\lambda-A_1(0)\Delta)[u]\|_{{\rm C}^{ \alpha}(\R^2)}\geq|\lambda|\cdot\|u\|_{{\rm C}^{\alpha}(\R^2)}+\|u\|_{{\rm C}^{2+\alpha}(\R^2)},
\qquad u\in {\rm h}^{2+\alpha}(\R^2),\, \re\lambda\geq\omega_1.
\]
  In particular it holds that 
  \[
\kappa_1\|(\lambda-A_1(0)\Delta)[h\circ|\,\cdot\,|]\|_{{\rm C}^{\alpha}(\R^2)}
\geq|\lambda|\cdot\|h\circ|\,\cdot\,|\|_{{\rm C}^{ \alpha}(\R^2)}+\|h\circ|\,\cdot\,|\|_{{\rm C}^{2+\alpha}(\R^2)}
\geq|\lambda|\cdot\|h \|_{{\rm C}^{\alpha}(\R)}+\|h \|_{{\rm C}^{2+\alpha}(\R)}
\]
for $ h\in {\rm h}_{e}^{2+\alpha}(\R) $ and $ \re\lambda\geq\omega_1.$
  Moreover, in virtue of
  \[
  \|(\lambda-A_1(0)\Delta)[h\circ|\,\cdot\,|]\|_{{\rm C}^{ \alpha}(\R^2)}=\|((\lambda-\bA_1)[h])\circ|\,\cdot\,|\|_{{\rm C}^{ \alpha}(\R^2)}
  =\|(\lambda-\bA_1)[h] \|_{{\rm C}^{ \alpha}(\R)}
  \]
we conclude that 
\begin{align}\label{EDI2}
\kappa_1\|(\lambda-\bA_1)[h]\|_{{\rm C}^{ \alpha}(\R)}\geq|\lambda|\cdot\|h \|_{{\rm C}^{ \alpha}(\R)}+\|h \|_{{\rm C}^{2+\alpha}(\R)},
\qquad h\in {\rm h}_e^{2+\alpha}(\R),\, \re\lambda\geq\omega_1.
\end{align}
The constants $\kappa_1$ and $\omega_1$ can be chosen such that \eqref{EDI2}  holds true also when replacing $\bA_1$ by $(A_1(\pi)/A_1(0))\bA_1=\tau_{-\pi}\bA_2\tau_\pi,$ where 
$ \tau_a,\, a\in\R$, denotes   the  right translation by $a$.

In particular \eqref{EQN}, \eqref{EDI1}, and \eqref{EDI2} ensure there exists $\kappa'\geq 1$ and $\omega'>0$ such that
\begin{align}\label{EDI3}
\kappa'\|(\lambda-\bA_i)[\pi_i^\e h]\|_{{\rm C}^\alpha(I)}\geq |\lambda|\cdot \|\pi_i^\e h\|_{{\rm C}^\alpha(I)}+\|\pi_i^\e h\|_{{\rm C}^{2+\alpha}(I)}
\end{align}
for all $  h\in {\rm h}_e^{2+\alpha}(\s),\, \re\lambda\geq \omega',\, 1\leq i\leq 3, $ and all $\e\in(0,\e_0]$.
The estimate \eqref{EDI3} together with the observation that the map 
\begin{align}\label{EDI5}
\Big[h\mapsto\sum_{i=1}^3\|\pi_i^\e h\|_{{\rm C}^{k+\alpha}(I)}\Big]:{\rm C}^{k+\alpha}(I)\to\R,\qquad k\in\N,\, 
\end{align} 
defines a norm on ${\rm C}^{k+\alpha}(I)$ which is equivalent to the standard H\"older norm are essential for establishing the following result.

\begin{lemma}\label{L32}
There exist $\kappa\geq  1$ and $\omega>0$ such that 
\begin{align}\label{EDI4}
\kappa\|(\lambda-\bA)[h]\|_{0}\geq |\lambda|\cdot \|h\|_{0}+\|h\|_{1}
\end{align}
for all $  h\in {\rm h}_e^{2+\alpha}(\s)$ and all $ \re\lambda\geq \omega$.
\end{lemma}
\begin{proof}
Letting $\kappa'\geq1$ and $\omega'>0$ denote the constants in \eqref{EDI3}, we  chose $\mu:=(2\kappa')^{-1}$ in Lemma~\ref{L31}.  
  Lemma~\ref{L31} together with  \eqref{EDI3} yields  
\begin{align*}
  \kappa' \|\pi_i^\e (\lambda-\bA) [h]\|_{{\rm C}^\alpha(I) }&\geq \kappa'\|(\lambda-\bA_i)[\pi^\e_i h]\|_{{\rm C}^\alpha(I)}-  \kappa' \|\pi_i^\e \bA [h]-\bA_i[\pi^\e_ih]\|_{{\rm C}^\alpha(I)}\\[1ex]
  &\geq \frac{1}{2 } \|\pi_i^\e h\|_{{\rm C}^{2+\alpha}(I)}+|\lambda|\cdot\|\pi_i^\e h\|_{{\rm C}^\alpha(I)}-\kappa'K\| h\|_{{\rm C}^2(I)}
 \end{align*}
 for $1\leq i\leq 3$, $  h\in {\rm h}_e^{2+\alpha}(\s) $, and $ \re\lambda\geq \omega'.$
 In virtue of \eqref{EDI5} and of \eqref{EQN} it now follows that there exists a constant $\kappa''\geq1$ such that 
  \begin{align*}
  \kappa'' \big(\| h\|_{{\rm C}^{2+\alpha/2}(\s)}+ \|(\lambda-\bA) [h]\|_{{\rm C}^\alpha(\s) }\big)&\geq \| h\|_{{\rm C}^{2+\alpha}(\s)}+|\lambda|\cdot\| h\|_{{\rm C}^\alpha(\s)}
 \end{align*}
  for $  h\in {\rm h}_e^{2+\alpha}(\s) $ and $ \re\lambda\geq \omega'.$
  Finally, the  interpolation property \eqref{Intprop},  the latter estimate, and
   Young's inequality ensure that there exist constants
  $\kappa\geq 1$ and $\omega>0$ such that \eqref{EDI4} is satisfied.
\end{proof}

To derive the desired generation result we are left to show that $\omega-\bA\in{\rm Isom}({\rm h}_{e}^{2+\alpha}(\s),{\rm h}_{e}^{\alpha}(\s)).$
To this end we infer from \eqref{EDI4} that $\omega-\bA$ is one-to-one.
Having shown that $\bA\in\kL({\rm h}_{e}^{2+\alpha}(\s),{\rm h}_{e}^{\alpha}(\s))$ is a Fredholm operator of index zero, the isomorphism property follows then in view of the compactness   
of the embedding $ {\rm h}_{e}^{2+\alpha}(\s)\hookrightarrow{\rm h}_{e}^{\alpha}(\s).$

\begin{lemma}\label{L33}
 $\bA\in\kL({\rm h}_{e}^{2+\alpha}(\s),{\rm h}_{e}^{\alpha}(\s))$ is a Fredholm operator of index  zero.
\end{lemma}
\begin{proof}
Since $A_1>0$, the equation $\bA[h]=0$ is equivalent to
\[
h''+\frac{h'}{\tan}=0, 
\] 
hence $ (h'\sin)'=0$. The kernel of $\bA$ consists thus only of constant functions.

It is easy to see that the range of $\bA$ is contained in 
$$Y:=\Big\{f\in{\rm h}_{e}^{\alpha}(\s)\,:\, \int_0^{\pi} \frac{f\sin}{A_1}\, dx=0  \Big\},$$
which is a closed subspace of ${\rm h}_{e}^{\alpha}(\s)$ of codimension $1$.
To show that the range of $\bA$ coincides with $Y$ we  associate to $f\in Y$ the function 
\[
h(x):=\int_0^x\frac{1}{\sin(t)}\int_0^t\frac{f\sin}{A_1}(s)\, ds\, dt, \qquad x\in[0,2\pi].
\]
Using  the property defining 
 $Y$, it is not difficult to check that $h$ is twice continuously differentiable with 
\begin{align*}
h(2\pi)&=\int_0^{2\pi}\frac{1}{\sin (t)}\int_0^t\frac{f\sin}{A_1}(s)\, ds\, dt
=-\int_0^{2\pi}\frac{1}{\sin (t)}\int_t^{\pi}\frac{f\sin}{A_1}(s)\, ds\, dt\\[1ex]
&=-\int_0^{\pi}\frac{1}{\sin (t)}\int_t^{\pi}\frac{f\sin}{A_1}(s)\, ds\, dt-\int_\pi^{2\pi}\frac{1}{\sin (t)}\int_t^{\pi}\frac{f\sin}{A_1}(s)\, ds\, dt=0=h(0).
\end{align*}
The second last identity above  follows by using appropriate  substitutions in the second integral.  
Moreover, it holds that  $h'(0)=h'(2\pi)=0$, $h''(0)=h''(2\pi)$, and  
\[
 A_1\Big(h''+\frac{h'}{\tan} \Big) =f\qquad\text{in $\R$},
\]
as we may extend $h$ by periodicity to $\R$. 
Some standard (but lengthy) arguments show that $ h'/\sin$ lies in $ {\rm h}_{e}^{\alpha}(\s)$, which implies that 
$h\in{\rm h}_{e}^{2+\alpha}(\s)$.
Thus, $f$ belongs to the range of $\bA$ and the claim follows.
\end{proof}

\begin{proof}[Proof of Theorem \ref{T:31}]
In view of Lemma \ref{L32}  it remains to show that $\omega-\bA:{\rm h}_{e}^{2+\alpha}(\s)\to {\rm h}_{e}^{\alpha}(\s)$ is an isomorphism.
This property is an immediate consequence of the estimate \eqref{EDI4}, which implies in particular  that $\omega-\bA$ is injective, and of the fact that $\omega-\bA$ is a Fredholm operator of index zero, cf. Lemma \ref{L33} (we recall at this point  that the embedding $ {\rm h}_{e}^{2+\alpha}(\s)\hookrightarrow {\rm h}_{e}^{\alpha}(\s)$ is compact). 
\end{proof}

 We conclude this section with the proof of the well-posedness result stated in Theorem \ref{MT1}.
 The proof of the parabolic smoothing property for the function $v$ is postponed to Section~\ref{Sec:4}.
 \begin{proof}[Proof of Theorem \ref{MT1}]
We first address the solvability of \eqref{TP}.
As a direct consequence of  Theorem~\ref{T:31} we have that 
  \[
 -\p\Phi(h_0,c_0,d_0)\in \kH(\E_1\times\R^2,\E_0\times \R^2)
 \] 
 for all $(h_0,c_0,d_0)\in\cO\times\R\times (0,\infty)$. 
 Recalling also  \eqref{REG} and the interpolation property of the small H\"older spaces \eqref{Intprop},
   the assumptions of \cite[Theorem 8.4.1]{L95} are all satisfied in the context of \eqref{TP}.
  Hence, for each  $(h_0,c_0,d_0)\in\cO\times \R\times (0,\infty)$,  \eqref{TP} possesses a unique maximal strict  solution $$(h,c,d):=(h,c,d)(\,\cdot\,;(h_0,c_0,d_0))$$
such that 
    \begin{align*}
&h\in  {\rm C}^1([0,t^+),\E_0)\cap {\rm C}([0,t^+),\cO),\\[1ex]
&c\in {\rm C}^1([0,t^+),\R),\\[1ex]
&d\in {\rm C}^1([0,t^+),(0,\infty)),
 \end{align*}
 where $t^+:=t^+(h_0,c_0,d_0)\in(0,\infty]$.
Since by assumption  $h_0:=v_0(c_0-d_0\cos) \in\cO$,  the existence and uniqueness claim in Theorem \ref{MT1} follows.
That  $a,\,b\in {\rm C}^\omega((0,t^+))$ is a straight forward consequence of \cite[Corollary 8.4.6]{L95}.
The real-analyticity property for $v$ (or $h$) is however more subtle and is established  in Section \ref{Sec:4} below.
 \end{proof}

\section{Parabolic smoothing}\label{Sec:4}
In the following we   consider a solution $(v,a,b)$ to \eqref{Pv} with maximal existence time $t^+$ as found in Theorem~\ref{MT1}. 
and we  prove that the associated function
  $$[(t,x)\mapsto h(t,x)]:(0,t^+)\times(0,\pi)\to\R $$  is  real-analytic. 
In this way we establish the parabolic smoothing property for the function $v$ as stated in Theorem~\ref{MT1}.
The proof below exploits a  parameter trick which has been used in other variants also in \cite{An90, AC11, ES96, PSS15, M16x} 
to improve the regularity of solutions to parabolic or elliptic equations. 
The degenerate parabolic setting considered herein raises new difficulties, in particular due to the fact that the solutions $h$ vanish at $0$ and $\pi$,
which hinder us to establish real-analyticity of $h$ in a neighborhood of these points.     

To start, we fix an arbitrary  constant $T$ such that   $0< T <t^+$. 
 Given    $\lambda\in\R$  with 
\begin{align}\label{SAL}
T|\lambda|<\min_{[0,T]} \frac{d(t)}{2}=:\vartheta_0
\end{align}  
and $t\in[0,T]$,
we introduce the function $\phi_\lambda(t):\R\to(a(t),b(t))$ with
\[
\phi_\lambda(t,x):=c(t)-d(t)\cos(x)+t\lambda\sin^2(x), \qquad t\in[0,T],\, x\in(0,\pi).
\]
The smallness condition \eqref{SAL} ensures that  $\phi_\lambda(t):(0,\pi)\to(a(t),b(t))$ is a real-analytic diffeomorphism.
We associate to $v$ the function $h(t,x,\lambda):=v(t,\phi_\lambda(t,x))$, $x\in\R$, $t\in[0,T]$, $|\lambda|<T^{-1}\vartheta_0 .$
Let further $ h(\lambda):=h(\,\cdot\,,\,\cdot\,,\lambda)$.
Since $h(t,x,0)=h(t,x)$ for $t\in[0,T]$ and $x\in\R$,  Theorem \ref{MT1} yields  
\[
h(0)\in  {\rm C}^1([0,T],{\rm h}^{\alpha}_e(\s))\cap {\rm C}([0,T],{\rm h}^{2+\alpha}_e(\s)).
\]
Clearly,   $h(\lambda)$  is $2\pi$-periodic and even with respect to $x$.
Observing that 
\[
h(t,x,\lambda)=h\Big(t,\arccos\Big(\cos x-\frac{t\lambda\sin^2(x)}{d(t)}\Big),0\Big),\qquad t\in[0,T],\,x\in\R,
\]
  tedious computations show that  
\begin{align}\label{Conv}
h(\lambda) \in  {\rm C}^1([0,T],{\rm h}^{\alpha}_e(\s))\cap {\rm C}([0,T],{\rm h}^{2+\alpha}_e(\s))\qquad\text{for all $|\lambda|<T^{-1}\vartheta_0$.}
\end{align}
We emphasize that Lemma \ref{L:A1} (ii) plays a key role in the proof of \eqref{Conv}.
 Moreover, given $t\in[0,T]$,  it holds that
\begin{align*}
\p_x^2 h(t,0,\lambda)> h(t,0,\lambda)=0=h(t,\pi,\lambda)<\p_x^2 h(t,\pi,\lambda)
\end{align*}
together with 
\[ h(t,x,\lambda)>0, \quad    x\in(0,\pi).\]
Furthermore,  the pair $(h(\lambda),c,d)$ solves the parameter dependent evolution problem
 \begin{align}\label{FNP'}
(\dot h,\dot c,\dot d)=\Psi(t,h, c,d,\lambda),\, \, t\in[0,T],\qquad (h(0),c(0), d(0))=(h_0,c_0,d_0),
\end{align}
where $\Psi:=(\Psi_1,\Psi_2,\Psi_3):[0,T]\times\cO\times \R\times (0,\infty)\times(-T^{-1}\vartheta_0,T^{-1}\vartheta_0)\to \E_0\times\R^2$ is defined by
 \begin{align*} 
&\hspace{-0.25cm}\Psi_1(t,h,c,d,\lambda)\\[1ex]
&:=    \frac{2(d+t\lambda \cos)^2}{(d+t\lambda\cos)^2-t^2\lambda^2}\cdot\frac{1}{2 (d+2t\lambda\cos)^2h+h'^2/\sin^2}\cdot \frac{h}{\sin^2}\\[1ex]
&\hspace{0.55cm} \times \Big[\frac{(d+t\lambda\cos)^2-t^2\lambda^2}{(d+2t\lambda\cos)^2}h''-\frac{d }{d+2t\lambda\cos}\frac{h'}{\tan}  
-\frac{t\lambda(2t^2\lambda^2-2d^2-3dt\lambda\cos+4t^2\lambda^2\cos^2)}{(d+2t\lambda\cos)^3}h'\sin\Big]\\[1ex]
&\hspace{0.55cm} -\frac{ h'^2/\sin^2}{2(d+2t\lambda\cos)^2h+ h'^2/\sin^2}-1\\[1ex]
&\hspace{0.55cm} +\Big[\frac{(d+2t\lambda)(d+d\cos-t\lambda\sin^2)}{d(d+2t\lambda\cos)h''(0)}-\frac{(d-2t\lambda)(d-d\cos+t\lambda\sin^2)}{d(d+2t\lambda\cos)h''(\pi)}\Big] \frac{h'}{\sin}\\[1ex]
&\hspace{0.55cm} +\frac{\lambda}{d+2t\lambda\cos}\Big[1+\frac{t}{d}\Big(\frac{d+2t\lambda}{h''(0)}+\frac{d-2t\lambda}{h''(\pi)}\Big)\Big]h'\sin 
\end{align*}
and
\begin{align*} 
 \Psi_2(t,h,c,d,\lambda)&:= \cfrac{d+2t\lambda }{h''(0)}-\cfrac{d-2t\lambda}{h''(\pi)},\\[1ex]
 \Psi_3(t,h,c,d,\lambda)&:=- \cfrac{d+2t\lambda}{h''(0)}-\cfrac{d-2t\lambda}{h''(\pi)}.
\end{align*}
 Recalling  \eqref{SAL}, it then follows that 
 \[
\Psi\in {\rm C}^\omega([0,T]\times\cO\times\R\times(0,\infty)\times (-T^{-1}\vartheta_0,T^{-1}\vartheta_0),\E_0\times\R^2). 
 \]
Observing that $[h\mapsto h'\sin]:\E_1\to\E_0$ is a bounded operator which can be estimated in a similar way  as the operators  in 
Lemma \ref{L22}, we may repeat the arguments in Sections \ref{Sec:2}-\ref{Sec:3} to conclude that 
 \[
-\p_{(h,c,d)}\Psi(t, h_0,c_0,d_0,\lambda)\in\mathcal{H}(\E_1\times\R^2,\E_0\times\R^2)
 \] 
  for all $(t, h_0,c_0,d_0,\lambda)\in [0,T]\times\cO\times\R\times(0,\infty)\times (-T^{-1}\vartheta_0,T^{-1}\vartheta_0).$
  Applying \cite[Theorem 8.4.1]{L95}, it follows that \eqref{FNP'} possesses for each $(h_0,c_0,d_0,\lambda)\in\cO\times\R\times(0,\infty)\times (-T^{-1}\vartheta_0,T^{-1}\vartheta_0)$ a unique maximal strict solution $(h,c,d)=(h,c,d)(\cdot; (h_0,c_0,d_0,\lambda))$ with
  \[
(h,c,d)\in {\rm C}^1([0,t^+),\E_0\times\R^2)
\cap C([0,t^+),\cO\times\R\times(0,\infty)),
  \]
  where  $t^+=t^+(h_0,c_0,d_0,\lambda)\in(0,T]$ is the maximal existence time.
 In view of \cite[Corollary 8.4.6]{L95} we may conclude that the  mapping
 \begin{align*}
 [(t,h_0,c_0,d_0,\lambda)\mapsto h(t;(h_0,c_0,d_0,\lambda))]:\Omega\to\E_1,
 \end{align*}
where
\[
\0:=\{(t,h_0,c_0,d_0,\lambda)\,:\,(h_0,c_0,d_0,\lambda)\in\cO\times\R\times(0,\infty)\times (-T^{-1}\vartheta_0,T^{-1}\vartheta_0),\, t\in(0,t^+)\}
\]
is real-analytic.
Let now $x_0\in(0,\pi)$ be fixed. Since $[a\mapsto a(x_0)]:\E_1\to\R$ is a real-analytic map, we obtain for the function $h$ determined by the solution $(v,a,b) $ considered above,  in particular that
\begin{align}\label{MP1}
\Big[(t,\lambda)\mapsto h\Big(t,\arccos\Big(\cos x_0-\frac{t\lambda\sin^2(x_0)}{d(t)}\Big)\Big)\Big]:(0,T)\times (-T^{-1}\vartheta_0,T^{-1}\vartheta_0)\to\R
\end{align}
 is real-analytic too.
Additionally, given $\tau\in(0,T)$, for sufficiently small $\delta>0$ it holds that
\begin{align}\label{MP2}
\Big[(t,x)\to\Big(t,\frac{(\cos(x_0)-\cos(x))d(t)}{t\sin^2(x_0)}\Big)\Big]:(\tau,T)\times (x_0-\delta,x_0+\delta)\to(0,T)\times (-T^{-1}\vartheta_0,T^{-1}\vartheta_0)
\end{align}
is well-defined and real-analytic. Here we use the real-analyticity of  $d$ in $(0,T)$ which we  already established. 
 Composing  the mappings \eqref{MP1} and \eqref{MP2}, it follows  in view of the fact that $x_0\in(0,\pi)$ is arbitrary that 
 \begin{align*}
[(t,x)\to h(t,x)]:(0,T)\times (0,\pi)\to\R
\end{align*}
is real-analytic. Recalling that $h(t,x)=v(t, c(t)-d(t)\cos(x)),$ the property
 $$v\in {\rm C}^\omega(\{(t,x)\,:\, 0<t<t^+,\, a(t)<x<b(t)\},(0,\infty))$$ follows at once.   

 %%%%%%%%%%%%%%%%%%%%%%%%%%%%%%%%%%%%%%%%%%%%%%%
%%%%%%%%%%%%%%%%%%%%%%%%%%%%%%%%%%%%%%%%%%%%%%
%%%%%%%%%%%%%%%%%%%%%%%%%%%%%%%%%%%%%%%%%%%%%%%
%%%%%%%%%%%%%%%%%%%%%%%%%%%%%%%%%%%%%%%%%%%%%%
%%%%%%%%%%%%%%%%%%%%%%%%%%%%%%%%%%%%%%%%%%%%%%%
%%%%%%%%%%%%%%%%%%%%%%%%%%%%%%%%%%%%%%%%%%%%%%
%%%%%%%%%%%%%%%%%%%%%%%%%%%%%%%%%%%%%%%%%%%%%%%
%%%%%%%%%%%%%%%%%%%%%%%%%%%%%%%%%%%%%%%%%%%%%%
\appendix
\section{}\label{S:A}
 %%%%%%%%%%%%%%%%%%%%%%%%%%%%%%%%%%%%%%%%%%%%%%%
%%%%%%%%%%%%%%%%%%%%%%%%%%%%%%%%%%%%%%%%%%%%%%
%%%%%%%%%%%%%%%%%%%%%%%%%%%%%%%%%%%%%%%%%%%%%%%
%%%%%%%%%%%%%%%%%%%%%%%%%%%%%%%%%%%%%%%%%%%%%%
%%%%%%%%%%%%%%%%%%%%%%%%%%%%%%%%%%%%%%%%%%%%%%%
%%%%%%%%%%%%%%%%%%%%%%%%%%%%%%%%%%%%%%%%%%%%%%
%%%%%%%%%%%%%%%%%%%%%%%%%%%%%%%%%%%%%%%%%%%%%%%
%%%%%%%%%%%%%%%%%%%%%%%%%%%%%%%%%%%%%%%%%%%%%%
The next result shows that the two approaches used in the Introduction to derive evolution equations for the functions $a$ and $b$ require the same assumptions. 
In particular, it shows that the solutions to the problem \eqref{Pv} describe closed ${\rm C}^2$-surfaces without boundary and with positive curvature at the points on the rotation axis.

\begin{lemma}\label{L:A-1}
Let $0<a$ and let $u\in {\rm C}([0,a])\cap  {\rm C}^2([0,a))$ satisfy $u(x)>0$ for all $x\in[0,a)$ and $u(a)=0$. Then, the following are equivalent:
\begin{itemize}
\item[(i)] $\underset{x\to a}\lim u'(x)=-\infty,$ $\displaystyle\underset{x\to a}\lim (uu')(x)<0,$ and 
 $$\exists\,\underset{x\to a}\lim \frac{u''}{u'^3}(x)\in \R\cup\{\pm\infty\}.$$
\item[(ii)] There exists $\e>0$ such that $u:[a-\e,a]\to [0,u(a-\e)]$ is invertible and the inverse $w:[0,u(a-\e)]\to [a-\e,a]$ satisfies $w\in {\rm C}^2([0,u(a-\e)])$,
 $w'(0)=0$, and $w''(0)<0$.\\[-1ex]
\item[(iii)] The function $v:=u^2/2$ satisfies $v\in {\rm C}^1([0,a])$, $v'(a)<0$, and $\underset{x\to a}\lim (vv'')(x)=0.$
 \end{itemize}
\end{lemma}
\begin{proof} We first prove the implication (i)$\implies$(ii).
It is obvious that if $\e>0$ is sufficiently small, then $u:[a-\e,a]\to [0,u(a-\e)]$ has an inverse function $w:[0,u(a-\e)]\to [a-\e,a]$ that satisfies
 $w\in {\rm C}([0,u(a-\e)])\cap {\rm C}^2((0,u(a-\e)]).$
Furthermore, it holds that $w'(0)=0$ and 
\begin{align}\label{esti1}
\lim_{y\to 0}\frac{w'(y)}{y}=\frac{1}{\underset{x\to a}\lim (uu')(x)}.
\end{align}
Hence, $w$ is twice differentiable in $0$ and $w''(0)<0$.
 Furthermore, the mean value theorem yields the existence of a sequence $y_n\to 0+$ such that $w''(y_n)\to w''(0)$. 
Since
\[
\lim_{y\to 0}w''(y)=\lim_{x\to a}w''(u(x))=-\lim_{x\to a}\frac{u''}{u'^3}(x),
\]
we obtain the following relation
\[
\lim_{x\to a}\frac{u''}{u'^3}(x)=-w''(0),
\]
and therewith we get that  $w\in  {\rm C}^2([0,u(a-\e)]).$

We now prove the implication (ii)$\implies$(iii). We may assume that  $w'(y)<0$ for $y>0$.  
Invoking \eqref{esti1} we get that    $v\in {\rm C}^1([0,a])$ and $v'(a)=1/w''(0)<0$.
Moreover it holds that
\begin{align*} 
2\underset{x\to a}\lim (vv'')(x)=\underset{x\to a}\lim (u^3u''+u^2u'^2)(x)=\underset{y\to 0}\lim \Big(-\frac{y^3w''(y)}{w'^3(y)}+\frac{y^2}{w'^2(y)}\Big)=0,
\end{align*} 
and this proves (iii).

We conclude with the proof of (iii)$\implies$(i). 
The relations   $\underset{x\to a}\lim u'(x)=\infty $ and $\displaystyle\underset{x\to a}\lim (uu')(x)<0 $ are immediate and together with 
\begin{align*}
\underset{x\to a}\lim\frac{u''}{u'^3}(x)=\underset{x\to a}\lim\Big(\frac{2vv''}{v'^3}(x)-\frac{1}{v'(x)}\Big)=-\frac{1}{v'(a)}
\end{align*}
we have completed the proof.
\end{proof}

Lemma \ref{L:A1} provides a continuity result which is used  to establish \eqref{Conv}.
 This lemma also exemplifies why  the small H\"older spaces are to be preferred in certain applications to the classical ones. 
 \begin{lemma}\label{L:A1} Let $\alpha\in(0,1)$.
 \begin{itemize}
 \item[(i)] Given $a\in{\rm C}^\alpha(\s), $ the mapping $$[b\mapsto a\circ b]: W^1_\infty(\s)\to {\rm C}^\alpha(\s)$$ is in general not continuous.\\[-1ex]
 \item[(ii)] Given $a\in{\rm h}^\alpha(\s), $ the mapping $$[b\mapsto a\circ b]: W^1_\infty(\s)\to {\rm h}^\alpha(\s)$$ is continuous.
 \end{itemize}
 \end{lemma}
 \begin{proof}
 It is easy to verify that  $[b\mapsto a\circ b]: W^1_\infty(\s)\to {\rm C}^\alpha(\s)$ is well-defined.
 The following example  shows that this (nonlinear) mapping is in general not continuous. 
 Indeed, let $\phi\in  {\rm C}^\infty_0(\R) $ be a function which satisfies $\phi=1$ on $[-1,1]$ and $\phi=0$ in $\R\setminus[-2,2]$.
 The $2\pi$-periodic extension   $a$ of
 \[[x\mapsto |x|^\alpha\phi(x)]:[-\pi,\pi]\to\R\]
 satisfies  $a\in{\rm C}^\alpha(\s).$ Given $1\leq n\in\N$, let $b_n$ and $b$ denote the $2\pi$-periodic extensions of
  \[\text{$[x\mapsto (x+1/n)\phi(x)]:[-\pi,\pi]\to\R$ \qquad and\qquad $[x\mapsto x\phi(x) ]:[-\pi,\pi]\to\R$}.\]
  It then holds $b,\, b_n\in W^1_\infty(\s)$ and $b_n\to b$ in $W^1_\infty(\s)$.
  Since
  \[
|a(b_n(-1/n))-a(b(-1/n))-(a(b_n(0))-a(b(0)))|  =\frac{2}{ n^\alpha}
  \]
  it follows that 
  \[
\|a\circ b_n-a\circ b\|_\alpha\geq[a\circ b_n-a\circ b]_\alpha\geq 2  \qquad\text{for all $n\in\N, \, n\geq 1$,}
  \]
  which proves (i).
  
  We now  prove (ii). 
  Let  thus 
 $b_n, b \in W^1_\infty(\s)$ with $b_n \to b$ as $n\to \infty $ in
$W^1_\infty(\s)$ and
$(a_m)\subset {\rm C}^\infty(\s)$ be a sequence with $a_m\to a$ in ${\rm C}^\alpha(\s).$
  It follows that $a_m\circ b_n\in  {\rm h}^\alpha(\s) $ for $n,\, m\in\N$. 
  Moreover, it holds
  \[
\|a\circ b_n-a\circ b\|_\alpha\leq\|a\circ b_n-a_m\circ b_n\|_\alpha +\|a_m\circ b_n-a_m\circ b\|_\alpha+\|a_m\circ b-a\circ b\|_\alpha,
  \]
  where 
  \begin{align*}
  \|a\circ b_n-a_m\circ b_n\|_\alpha&\leq (1+\|b_n'\|_0^\alpha)\|a_m-a\|_\alpha,\\[1ex]
  \|a_m\circ b-a\circ b\|_\alpha&\leq (1+\|b'\|_0^\alpha)\|a_m-a\|_\alpha,\\[1ex]
  \|a_m\circ b_n-a_m\circ b\|_\alpha&\leq\|a_m\|_\alpha\|b_n-b\|_0^\alpha+\|a_m'\|_\alpha\|b_n-b\|_\alpha (1+\|b'\|_0^\alpha+\|b_n'\|_0^\alpha).
  \end{align*}
  These estimates show that $a\circ b_n\to a\circ b$  in ${\rm C}^\alpha(\s) $ and that each ball in ${\rm C}^\alpha(\s)$ centered in $a\circ b$ contains a 
  function   $a_m\circ b_n$ with $n,\, m\in\N$ suitably large.
  Since $a_m\circ b_n \in  {\rm h}^\alpha(\s) $, it follows  that also $a\circ b\in  {\rm h}^\alpha(\s),$ and this completes
   the proof.
 \end{proof}

  %%%%%%%%%%%%%%%%%%%%%%%%%%%%%%%%%%%%%%%%%%%%%%%
%%%%%%%%%%%%%%%%%%%%%%%%%%%%%%%%%%%%%%%%%%%%%%
%%%%%%%%%%%%%%%%%%%%%%%%%%%%%%%%%%%%%%%%%%%%%%%
%%%%%%%%%%%%%%%%%%%%%%%%%%%%%%%%%%%%%%%%%%%%%%
%%%%%%%%%%%%%%%%%%%%%%%%%%%%%%%%%%%%%%%%%%%%%%%
%%%%%%%%%%%%%%%%%%%%%%%%%%%%%%%%%%%%%%%%%%%%%%
%%%%%%%%%%%%%%%%%%%%%%%%%%%%%%%%%%%%%%%%%%%%%%%
%%%%%%%%%%%%%%%%%%%%%%%%%%%%%%%%%%%%%%%%%%%%%%

\bibliographystyle{abbrv}
\bibliography{GM}
\end{document}